\documentclass[reqno]{amsart}

\usepackage{times}
\usepackage{mathrsfs}
\usepackage{amsmath}
\usepackage{amsthm}
\usepackage{amsfonts}

\newtheorem{thm}{Theorem}[section]
\newtheorem{lem}[thm]{Lemma}

\newtheorem{prop}[thm]{Proposition}
\newtheorem{rems}[thm]{Remarks}
\newtheorem{rem}[thm]{Remark}

\DeclareMathAlphabet{\mathpzc}{OT1}{pzc}{m}{it}

\numberwithin{equation}{section}

\newcommand{\Wqb}{W_{q,\mathcal{B}}}
\newcommand{\Wpb}{W_{p,\mathcal{B}}}
\newcommand{\Bqb}{B_{q,p;\mathcal{B}}}
\newcommand{\R}{\mathbb{R}}

\newcommand{\N}{\mathbb{N}}

\newcommand{\Ac}{\mathcal{A}}
\newcommand{\E}{\mathbb{E}}
\newcommand{\B}{\mathcal{B}}
\newcommand{\ml}{\mathcal{L}}

\newcommand{\Om}{\Omega}

\newcommand{\ve}{\varepsilon}
\newcommand{\rd}{\mathrm{d}}
\newcommand{\divv}{\mathrm{div}}

\newcommand{\bqn}{\begin{equation}}
\newcommand{\eqn}{\end{equation}}
\newcommand{\bqnn}{\begin{equation*}}
\newcommand{\eqnn}{\end{equation*}}
\newcommand{\bear}{\begin{eqnarray}} 
\newcommand{\eear}{\end{eqnarray}} 
\newcommand{\bean}{\begin{eqnarray*}} 
\newcommand{\eean}{\end{eqnarray*}} 
\newcommand{\bs}{\begin{split}}
\newcommand{\es}{\end{split}}

\newcommand{\dhr}{\mathrel{\lhook\joinrel\relbar\kern-.8ex\joinrel\lhook\joinrel\rightarrow}}

\title[Bifurcation in structured population models]
{Bifurcation of Positive Equilibria in Nonlinear Structured Population Models with Varying Mortality Rates}

\author[Ch. Walker]{Christoph Walker}

\address{Leibniz Universit\"at Hannover, Institut f\"ur Angewandte Mathematik, Welfengarten 1, D--30167 Hannover, Germany.}
\email{walker@ifam.uni-hannover.de}

\begin{document}

\begin{abstract}
A parameter-dependent model involving nonlinear diffusion for an age-structured population is studied. The parameter measures the intensity of the mortality. A bifurcation approach is used to establish existence of positive equilibrium solutions.
\end{abstract}


\maketitle

%%%%%%%%%%%%%%%%%%%%%%%%%%%%%%%%%%%
\section{Introduction}
%%%%%%%%%%%%%%%%%%%%%%%%%%%%%%%%%%%

Let $u=u(t,a,x)\ge 0$ be the distribution density at time $t\ge 0$ of individuals of a population structured by age $a\in J:=[0,a_m]$ and spatial position $x\in\Om$, where $a_m\in (0,\infty)$ denotes the maximal age and $\Om$ is a bounded and smooth domain in $\R^N$. Suppose that the individual's movement is governed by a nonlinear diffusion term $\divv_x(D(U(t,x),a)\nabla_xu)$ with dispersal speed $D(U,a)>0$ depending on age and on the total local population $$
U(t,x):=\int_0^{a_m}u(t,a,x)\rd a\ .
$$ 
Let $b=b(U,a)\ge 0$ and $\bar{\mu}=\bar{\mu}(U,a)\ge 0$ denote respectively the density dependent birth and death rate. Then a simple model describing the evolution of the population with initial distribution $u^0=u^0(a,x)\ge 0$ is
\begin{align*}
&\partial_t u+\partial_a u=\divv_x\big(D(U(t,x),a)\nabla_xu\big)-\bar{\mu}(U(t,x),a)u\ ,& t>0\, ,\ a\in J\, ,\ x\in\Om\ ,\\
&u(t,0,x)=\int_0^{a_m} b\big(U(t,x),a\big)u(a)\,\rd a\ ,& t>0\, , \ x\in\Om\ ,\\
&\delta u(t,a,x)+(1-\delta)\partial_\nu u(t,a,x)=0\ ,& t>0\, , \ a\in J\, ,\ x\in\partial\Om\ ,\\
&u(0,a,x)=u^0(a,x)\ ,&  a\in J\, , \ x\in\Om\ ,
\end{align*}
where either $\delta=1$ or $\delta=0$ corresponding to Dirichlet or Neumann boundary conditions. Models of this type have a long history and we refer to \cite{WebbSpringer} for a survey of structured population models.
The well-posedness of these equations and related population models involving nonlinear diffusion is investigated e.g. in \cite{WalkerDCDS}. Questions regarding the large time behavior are linked to equilibrium solutions. In this paper we thus shall focus on nontrivial nonnegative equilibrium solutions for such equations, that is, on time-independent solutions $u=u(a,x)\ge 0$ with $u\not\equiv 0$. 

Positive equilibrium solutions for age-structured population models without diffusion are studied e.g. in \cite{Cushing1,Cushing2,Cushing3} using bifurcation techniques or also in \cite{Pruess2} using fixed point theorems in conical shells. A bifurcation approach to age-structured population models with linear diffusion and linear birth but nonlinear death rates is used in \cite{DelgadoEtAl,DelgadoEtAl2}. For an approach to age-structured models including both nonlinear diffusion and nonlinear death and birth rates we refer to \cite{WalkerSIMA,WalkerJDE}, where local and global bifurcation, respectively, is shown for a bifurcation parameter measuring the intensity of the fertility similarly as in \cite{Cushing1,Cushing2,Cushing3}. The aim of this paper is to demonstrate that also the intensity of the mortality can be treated as bifurcation parameter in age-structured models with nonlinear diffusion. Moreover, as expected and opposed to the results of e.g. \cite{WalkerSIMA}, where the fertility intensity varies, in the present situation subcritical bifurcation occurs under realistic assumptions.\\

The approach we choose is based on introducing a parameter $\lambda$ measuring the intensity of the mortality without changing its structure; that is, we shall consider parameter-dependent death rates of the form \mbox{$\bar{\mu}=\lambda\mu(U,a)$} with $\mu=\mu(U,a)$ being a fixed reference function. Thus we are looking for solutions $u=u(a,x)$ to the parameter-dependent problem
\begin{align}
&\partial_a u=\divv_x\big(D(U(x),a)\nabla_xu\big)-\lambda\mu(U(x),a)u\ ,& a\in J\, ,\ x\in\Om\ ,\label{2a}\\
&u(0,x)=\int_0^{a_m} b\big(U(x),a)u(a)\,\rd a\ ,&  x\in\Om\ ,\label{2b}\\
&\delta u(a,x)+(1-\delta)\partial_\nu u(a,x)=0\ ,& a\in J\, ,\ x\in\partial\Om\ .\label{2c}
\end{align}
Clearly, $u\equiv 0$ is a solution to \eqref{2a}-\eqref{2c} for any value of $\lambda$. The main goal is then to establish existence of nontrivial solutions which are also nonnegative. Under suitable assumptions we shall prove that the theorem of Crandall-Rabinowitz \cite{CrandallRabinowitz} applies so that there is a unique value $\lambda_0>0$ for which a nontrivial branch $\{(\lambda,u_\lambda);\vert\lambda-\lambda_0\vert\ \text{small}\}$ bifurcates from the trivial branch $\{(\lambda,0);\lambda\in\R\}$
at the critical point $(\lambda_0,0)$ and that at least one part of the nontrivial branch near the critical point consists of nonnegative solutions. 

To be more precise, let $\sigma_1$ be the first eigenvalue of $-\Delta_x$ on $\Om$ subject to Dirichlet (if $\delta=1$) or Neumann (if $\delta=0$) boundary conditions, hence $\sigma_1>0$ in the first and $\sigma_1=0$ in the second case. Suppose that
 \bqn\label{3}
 \int_0^{a_m}b(0,a)e^{-\sigma_1\int_0^a D(0,r)\rd r }\,\rd a\, >\, 1\qquad\text{and}\qquad \mu(0,a)>0\ \text{for}\ a\ \text{near}\ 0\ .
 \eqn
Roughly speaking, the first assumption in \eqref{3} may be interpreted as that for a zero death rate, the population is (locally) increasing. Letting $\lambda_0>0$ be such that
$$
\int_0^{a_m}b(0,a)\, e^{-\lambda_0\int_0^a\mu(0,r)\rd r}\, e^{-\sigma_1\int_0^a D(0,r)\rd r}\,\rd a=1\ ,
$$
the following result on local bifurcation holds for equations \eqref{2a}-\eqref{2c}:

\begin{thm}\label{P}
Let $D\in C^{\infty,1}(\R\times J)$ with $D(z,a)\ge d_0>0$ for $z\in\R$ and $a\in J$. Further, let \mbox{$\mu, b\in C^{\infty,1}(\R\times J)$} be nonnegative and suppose \eqref{3}. Then $(\lambda_0,0)$ is a bifurcation point for \eqref{2a}-\eqref{2c}, that is, there is a unique local branch of nontrivial nonnegative solutions 
$$
(\lambda,u)\quad\text{in}\quad \R^+\times \big( C(J,C(\bar{\Om}))\cap C^{1}(\dot{J},C(\bar{\Om}))\cap C(\dot{J},C^2(\Om))\big)
$$ 
emanating from the critical point $(\lambda_0,0)$, where $\dot{J}:=J\setminus\{0\}$. In addition, if $\delta=0$ and
\bqn\label{4}
b(z,a)\le b(0,a)\ ,\quad \mu(z,a)\ge \mu(0,a)\ ,\quad z\ge 0\ ,\quad a\in J\ ,
\eqn
then bifurcation is subcritical, i.e. $\lambda\le \lambda_0$ for any nonnegative solution $(\lambda,u)$.
\end{thm}

Assumption \eqref{4} is a common modeling assumption stating that effects of population densities do neither increase fertility nor decrease mortality. The result thus shows that lowering the intensity of mortality below a critical value leads to other equilibrium solutions than the trivial one. We also refer to Section~\ref{sec3} for an example where subcritical bifurcation occurs when $\delta=1$ and \eqref{4} holds.

We shall emphasize that Theorem~\ref{P} is merely a consequence of the considerably more general Theorem~\ref{F} that includes general nonlinear elliptic diffusion operators not necessarily in divergence form (and also less regular data). The proof of Theorem~\ref{P} will be given as an application of Theorem~\ref{F} in Section~\ref{sec3}.\\
To cover a great variety of applications we thus shall consider \eqref{2a}-\eqref{2c} as an abstract equation of the form
\begin{align}
& \partial_au \, +\,     A(u,a)\,u\,=\,-\lambda h(a)u\,+\,g(\lambda,u,a)u\ , \quad  a\in J\ ,\label{5a}\\ 
&u(0)\, =\, \int_0^{a_m}b(u,a)\, u(a)\, \rd a\ ,& \label{5b}
\end{align}
in an ordered Banach space $E_0$ with positive cone $E_0^+$ for the unknown function $u:J\rightarrow E_0^+$. Here, $A(u,a)$ defines for fixed $(u,a)\in E_0\times J$ a bounded linear operator from a subspace $E_1$ of $E_0$ into $E_0$. Problem \eqref{2a}-\eqref{2c} then fits into this abstract framework by choosing $$E_0:=L_q(\Om)\ ,\quad E_1:=\{v\in W_q^2(\Om);\delta v+(1-\delta)\partial_\nu v=0 \ \text{on}\ \partial\Omega\}$$ for some $q\in (1,\infty)$ (where boundary values are interpreted in the sense of traces) and letting 
$$
A(u,a)w:=-\divv_x\big(D(U,a)\nabla_x w\big)\ ,\quad h(a):=\mu(0,a)\ ,\quad\text{and}\quad g(\lambda,u,a):=\lambda (\mu(U,a)-\mu(0,a))\ .
$$ 
In Section \ref{sec2} we consider the abstract equations \eqref{5a}, \eqref{5b} and prove under suitable assumptions in Theorem~\ref{F} a local bifurcation result. In Section 3 we give applications of Theorem~\ref{F} and prove in particular Theorem~\ref{P}. Finally, the appended Section~\ref{appendix} contains a result on the differentiability of superposition operators in Sobolev-Slobodeckii spaces used for the applications in Section~\ref{sec3} that we were unable to find in the literature in this form.

\section{The abstract bifurcation result}\label{sec2}

Studying the nonlinear problem \eqref{5a}, \eqref{5b} demands an investigation of its linearization around $u=0$. We first state the precise assumptions required.

\subsection{Assumptions}

Given Banach spaces $E$ and $F$ we let $\mathcal{L}(E,F)$ denote the space of all bounded and linear operators from $E$ into $F$, and $\mathcal{L}(E):=\mathcal{L}(E,E)$. We write $\mathcal{L}is(E,F)$ for the subspace of $\mathcal{L}(E,F)$ consisting of all topological isomorphism and $\mathcal{K}(E,F)$ for the subspace of compact operators. 

For the remainder of this section let $J:=(0,a_m)$ with $a_m\in (0,\infty]$ and note that $J$ may be bounded or unbounded. Moreover, we fix an ordered Banach space $E_0$ with positive cone $E_0^+$ and a dense subspace $E_1$ thereof which is also supposed to be compactly embedded in $E_0$. This latter property we express by writing $E_1\dhr E_0$. Given $\theta\in[0,1]$ and an admissible interpolation functor $(\cdot,\cdot)_\theta$ we equip the interpolation space $E_\theta:=(E_0,E_1)_\theta$ with the order induced by the positive cone $E_\theta^+:=E_\theta\cap E_0^+$. Note that \mbox{$E_\vartheta\dhr E_\theta$} for \mbox{$0\le \theta<\vartheta\le 1$} according to \cite[I.Thm.2.11.1]{LQPP}. In particular, we fix $p\in (1,\infty)$ and set $E_\varsigma:=(E_0,E_1)_{1-1/p,p}$; that is, $E_\varsigma$ is the real interpolation space between $E_0$ and $E_1$ of exponent $\varsigma:=1-1/p$. We then assume that
\bqn\label{6}
\mathrm{int}(E_\varsigma^+)\not=\emptyset\ ,
\eqn
where $\mathrm{int}(E_\varsigma^+)$ denotes the topological interior of the cone $E_\varsigma^+$. We set
$$
\E_0:=L_p(J,E_0)\qquad\text{and}\qquad \E_1:=L_p(J,E_1)\cap W_p^1(J,E_0)
$$
and recall that $\E_1\hookrightarrow BUC(J,E_\varsigma)$ (see \cite{LQPP}). Thus, the trace operator $\gamma_0u:=u(0)$ for $u\in\E_1$ is a well-defined operator $\gamma_0\in\mathcal{L}(\E_1,E_\varsigma)$. We also set $\E_1^+:=\E_1\cap L_p^+(J,E_0)$. Suppose that
\bqn\label{7}
\begin{aligned}   & F\ \text{is a Banach space ordered by a positive cone}\ F^+\ \text{with}\ F\cdot E_1\hookrightarrow E_\varsigma\ \text{and}\ F^+\cdot E_0^+\hookrightarrow E_0^+\ ,
\end{aligned}
\eqn
where e.g. $F\cdot E_1\hookrightarrow E_\varsigma$ means a continuous bilinear mapping (i.e. a {\it multiplication}) $F\times E_1\rightarrow E_\varsigma$, $(f,e)\mapsto f\cdot e$. Let $\Sigma$ be a fixed ball in $\E_1$ centered at 0 of some positive radius and assume that
\bqn\label{8}
g\in C^1(\R^+\times\Sigma,L_\infty(J,F))\ \text{with}\ g(\lambda,0)\equiv 0\ \text{for}\ \lambda\in\R^+\ 
\eqn
and
\bqn\label{9}
h\in L_1^+(J,\R)\cap L_\infty(J,\R)\ \text{with}\ h>0\ \text{near}\ a=0\ .
\eqn
Observe that \eqref{8} guarantees that we may interpret $g(\lambda,u)$ as an element of $L_\infty(J,\mathcal{L}(E_1,E_0))$ for $(\lambda,u)$ in $\R^+\times\Sigma$ fixed. Suppose then that
\bqn\label{10}
A\in C^1\big(\Sigma,L_\infty(J,\mathcal{L}(E_1,E_0))\big)
\eqn
is such that 
\bqn\label{11}
\begin{aligned}   & A(u)+\lambda h-g(\lambda,u)\in L_{\infty}(J,\ml(E_1,E_0))\ \text{generates a positive parabolic}\\
&\text{evolution operator}\ \Pi_{(\lambda,u)}(a,\sigma), 0\le\sigma\le a<a_m,\ \text{on}\ E_0\ \text{with regularity subspace}\\\
& E_1\ \text{for each}\ (\lambda,u)\in \R^+\times\Sigma\ .
\end{aligned}
\eqn
We refer to \cite{LQPP} for a definition and properties of parabolic evolution operators. Note that, due to  \eqref{8} and \eqref{9}, the parabolic evolution operator $\Pi_0:=\Pi_{(0,0)}$ is simply generated by $A(0)$ and
\bqn\label{11A}
\Pi_{(\lambda,0)}(a,\sigma)=e^{-\lambda\int_\sigma^a h(r)\rd r}\, \Pi_0(a,\sigma)\ ,\quad 0\le \sigma\le a<a_m\ .
\eqn
We further suppose that
\bqn\label{11B}
\text{there are}\ \zeta\in\R, \rho,\omega>0, \kappa\ge 1\ \text{such that}\ \zeta+ A(0)\in C^\rho(J,\mathcal{H}(E_1,E_0;\kappa,\omega))
\eqn
and that
\bqn\label{11C}
\begin{aligned}   & A(0)+\lambda h\in L_{\infty}(J,\ml(E_1,E_0))\ \text{possesses maximal}\ L_p\text{-regularity on}\ J\ ,\\
&\text{that is,}\ \big(\partial_a+A(0)+\lambda h,\gamma_0\big)\in\mathcal{L}is(\E_1,\E_0\times E_\varsigma),\ \text{for each}\ \lambda>0\ .
\end{aligned}
\eqn
We refer again to \cite{LQPP} for a definition of the space $\mathcal{H}(E_1,E_0;\kappa,\omega)$ and details about operators having maximal $L_p$-regularity. We agree upon the notation $A(u,a):=A(u)(a)$ and e.g. $(hu)(a):=h(a)u(a)$ for $a\in J$ and $u\in\E_1$. We point out that, owing to \eqref{11}, \eqref{11C}, and \cite[III.Prop.1.3.1]{LQPP}, the linear problem
$$
\partial_au \, +\,  \big(A(0,a)+\lambda h(a)\big)\,u\, =\, f(a)\ ,\quad a\in J\ ,\qquad
u(0)\, =\, u^0
$$
admits for each datum $(f,u^0)\in \E_0\times E_{\varsigma}$ and $\lambda> 0$ a unique solution $u\in \E_1$ given by
\bqn\label{u1}
u(a)=\Pi_{(\lambda,0)}(a,0)u^0+\int_0^a \Pi_{(\lambda,0)}(a,\sigma)f(\sigma)\,\rd \sigma\ ,\quad a\in J\ ,
\eqn
satisfying for some $c_0=c_0(\lambda)>0$
\bqn\label{u2}
\| u\|_{\E_1}\le c_0\big(\|f\|_{\E_0}+\|u^0\|_{E_{\varsigma}}\big)\ .
\eqn
Moreover, invoking \cite[II.Lem.5.1.3]{LQPP} it follows from \eqref{11B} that there are $M_0\ge1$ and $\omega_0\in\R$ such that for $0\le\gamma<\beta<\alpha\le 1$
\bqn\label{12}
\|\Pi_0(a,\sigma)\|_{\mathcal{L}(E_\gamma)}+(a-\sigma)^{\alpha-\gamma}\|\Pi_0(a,\sigma)\|_{\mathcal{L}(E_\beta,E_\alpha)}\le M_0e^{\omega_0 (a-\sigma)}\ ,\quad 0\le\sigma<a<a_m\ .
\eqn
We also assume that
\bqn\label{13}
\Pi_0(a,0)\ \text{is strongly positive for each}\ a\in (0,a_m)\ ,
\eqn
that is, $\Pi_0(a,0)\phi\in\mathrm{int}(E_\varsigma^+)$ for $\phi\in E_\varsigma^+\setminus\{0\}$ and $a\in (0,a_m)$, and that
\bqn\label{14}
\Pi_0(a,0)\Pi_0(\sigma,0)=\Pi_0(\sigma,0)\Pi_0(a,0)\ ,\quad 0\le a,\sigma<a_m\ .
\eqn
The latter condition means that the operators $\{A(0,a);a\in J\}$ commute with each other. Finally, we assume that
\bqn\label{15}
\begin{aligned}   & b\in C^1(\Sigma,L_{p'}^+(J,F))\ \text{with}\ 0\not\equiv b_0:=b(0,\cdot)\in L_{p'}^+(J,\R)\\
&\text{and} \ \int_0^{a_m}b_0(a)e^{\omega_0a}\,\rd a<\infty\ ´,
\end{aligned}
\eqn
where $p'$ is the dual exponent of $p$, i.e. $1/p+1/p'=1$. The last condition in \eqref{15} is obviously superfluous if $a_m<\infty$ or $\omega_0<0$.

For the remainder of this section we assume that conditions \eqref{6}-\eqref{11}, \eqref{11B}, \eqref{11C}, \eqref{13}-\eqref{15} hold and refer to Section \ref{sec3} for examples where these conditions are met. In particular, they hold in case of Theorem \ref{P}.\\

\subsection{The linear problem}

We begin by investigating the linearization of \eqref{5a}, \eqref{5b} around $u=0$, that is, by investigating the problem
\begin{align}
& \partial_au \, +\, \big(A(0,a)+\lambda h(a)\big)u\,=\,0\ , \quad  a\in J\ ,\label{16a}\\ 
&u(0)\, =\, \int_0^{a_m}b_0(a)\, u(a)\, \rd a\ \label{16b}\ .
\end{align}
It readily follows from the previous observations that any solution $(\lambda,u)\in \R^+\times\E_1$ of \eqref{16a}, \eqref{16b} is of the form
\bqn\label{16B}
u(a)=e^{-\lambda\int_0^a h(r)\rd r}\,\Pi_0(a,0)\,u(0)\ ,\quad a\in J\ ,\qquad u(0)=Q_\lambda u(0)\ ,
\eqn
where the operator $Q_\lambda$ is given by
$$
Q_\lambda:=\int_0^{a_m}b_0(a)\, e^{-\lambda\int_0^a h(r)\rd r}\, \Pi_0(a,0)\,\rd a\ 
$$
and enjoys the following properties:

\begin{lem}\label{H}
For any $\lambda\ge 0$, $Q_\lambda\in\mathcal{K}(E_\varsigma)$ is strongly positive. Hence, the spectral radius $r(Q_\lambda)>0$ of $Q_\lambda$ is a simple eigenvalue of $Q_\lambda$ and of its dual operator $Q_\lambda'$ with eigenvector $B_\lambda\in\mathrm{int}(E_\varsigma^+)$ and strictly positive eigenfunctional $B_\lambda'\in E_\varsigma'$, respectively. Moreover, $r(Q_\lambda)$ is the only eigenvalue with a positive eigenvector.
\end{lem}

\begin{proof}
Let $\theta\in (0,1/p)$ and set $\vartheta:=\theta+\varsigma>\varsigma$. Then we derive from \eqref{12} and \eqref{15} that $Q_\lambda\in\mathcal{L}(E_\varsigma,E_\vartheta)$ for $\lambda\ge 0$ and thus, since $E_\vartheta\dhr E_\varsigma$ by \cite[I.Thm.2.11.1]{LQPP}, we have 
$Q_\lambda\in\mathcal{K}(E_\varsigma)$, $\lambda\ge 0$. Hence the assertion follows from the Krein-Rutman theorem (e.g., see \cite[Thm.12.3]{DanersKochMedina}) and assumption \eqref{6} provided we can show that $Q_\lambda\in\mathcal{K}(E_\varsigma)$ is strongly positive. To fill this gap let $f'$ be any nontrivial element of the dual space $E_\varsigma'$ of $E_\varsigma$ with $\langle f',\phi\rangle_{E_\varsigma}\ge 0$ for $\phi\in E_\varsigma^+$. Let $\varphi\in E_\varsigma^+\setminus\{0\}$. Then it follows from \cite[Prop.A.2.7, Prop.A.2.10]{ClementEtAl} and \eqref{6} that \mbox{$\langle f',\Pi_0(a,0)\varphi\rangle_{E_\varsigma}> 0$} for $a\in (0,a_m)$ since $\Pi_0(a,0)\varphi\in\mathrm{int}(E_\varsigma^+)$ by \eqref{13}, and thus
$$
\langle f',Q_\lambda\varphi\rangle_{E_\varsigma}=\int_0^{a_m}b_0(a)\, e^{-\lambda\int_0^a h(r)\rd r}\, \langle f',\Pi_0(a,0)\varphi\rangle_{E_\varsigma}\,\rd a\, >\, 0
$$
owing to \eqref{15}. Hence $Q_\lambda\varphi$ is an interior point of $ E_\varsigma^+$ again due to \cite[Prop.A.2.7, Prop.A.2.10]{ClementEtAl} and assumption \eqref{6}. This yields the strong positivity of $Q_\lambda$.
\end{proof}

We assume in the sequel that
\bqn\label{18}
 r(Q_{0}) >1\ .
\eqn
Observe that \eqref{16B} implies that $u(0)$ is (if nonzero) an eigenvector of $Q_\lambda$ to the eigenvalue 1. If $u$ is nonnegative, i.e. $u\in\E_1^+$, then necessarily $u(0)\in E_\varsigma^+$ and so $r(Q_\lambda)=1$ by the previous lemma. The next lemma shows that $r(Q_\lambda)$ is strictly decreasing in $\lambda$. Hence, if \eqref{18} does not hold, there is no admissible (i.e. positive) value of $\lambda$ for which the linearized problem \eqref{16a}, \eqref{16b} admits a nonnegative nontrivial solution. The interpretation of the operator $Q_\lambda$ is that it contains information about the spatial distribution of the expected number of newborns that a population produces when the birth and death processes are described by $b(0,\cdot)$ and $\lambda h=\lambda \mu(0,\cdot)$, respectively, and spatial movement is governed by $A(0,\cdot)$. Hence, at equilibrium these processes yield exact replacement. Roughly speaking, assumption \eqref{18} may be interpreted as that the population subject to birth processes and spatial dispersal increases locally if $\lambda=0$, that is, if no deaths occur.

Under assumption \eqref{18}, the following lemma guarantees the existence of a unique value $\lambda_0>0$ with $r(Q_{\lambda_0})=1$.\\

The following auxiliary result uses the ideas of \cite{DelgadoEtAl2}:

\begin{lem}\label{C}
The mapping $[\lambda\mapsto r(Q_\lambda)]: [0,\infty)\rightarrow (0,\infty)$ is continuous, strictly decreasing, and \mbox{$\lim_{\lambda\rightarrow\infty}r(Q_\lambda)=0$}. In particular, there is a unique $\lambda_0\in(0,\infty)$ with $r(Q_{\lambda_0})=1$.
\end{lem}

\begin{proof}
Since $b_0\ge 0$ with $b_0\not\equiv 0$ and $h\ge 0$ with $h>0$ near $a=0$, it readily follows from \eqref{13} analogously to the proof of Lemma \ref{H} that $Q_\lambda -Q_\xi$ is strongly positive for $\xi>\lambda\ge 0$, that is,
\bqn\label{18b}
(Q_\lambda -Q_\xi)\phi\in\mathrm{int}(E_\varsigma^+)\ ,\quad \phi\in E_\varsigma^+\setminus\{0\}\ ,\quad \xi>\lambda\ge 0\ .
\eqn
Given $\lambda\ge 0$, let $B_\lambda\in\mathrm{int}(E_\varsigma^+)$ and $B_\lambda'\in E_\varsigma'$ be the eigenvectors and strictly positive eigenfunctionals introduced in Lemma~\ref{H}. Then, for $\xi>\lambda\ge 0$, we deduce from \eqref{18b} that
$$
r(Q_\lambda)\,\langle B_\lambda',B_\xi\rangle_{E_\varsigma}\, =\,\langle Q_\lambda'B_\lambda', B_\xi\rangle_{E_\varsigma}\, =\, \langle B_\lambda',Q_\lambda B_\xi\rangle_{E_\varsigma}\,>\,\langle B_\lambda',Q_\xi B_\xi\rangle_{E_\varsigma} =\, r(Q_\xi)\, \langle B_\lambda',B_\xi\rangle_{E_\varsigma} \ ,
$$
whence $r(Q_\lambda)>r(Q_\xi)$ so that $[\lambda\mapsto r(Q_\lambda)]$ is strictly decreasing. Next, let $\lambda> 0$ (the case $\lambda=0$ is analogous) and consider a sequence $(\lambda_j)$ such that $0\le\lambda_j\rightarrow\lambda$. Given $\varepsilon>0$ sufficiently smal we may assume that $0\le\lambda-\varepsilon<\lambda_j<\lambda+\varepsilon$ for all $j\in \N$. Note then that \eqref{9} implies
$$
Q_{\lambda-\varepsilon}\, B_\lambda\le e^{\varepsilon \|h\|_1}\,Q_\lambda\, B_\lambda\,=\, e^{\varepsilon \|h\|_1}\, r(Q_\lambda)\, B_\lambda
$$
with $\|h\|_1$ denoting the $L_1$-norm of $h$. Since $B_\lambda\in E_\varsigma^+$ we derive
\bqn\label{21}
r(Q_\lambda)\,e^{\varepsilon \|h\|_1}\, >\, r(Q_{\lambda-\varepsilon})
\eqn
from \cite[Cor.12.4]{DanersKochMedina} and \eqref{6}. Conversely, we have
$$
Q_{\lambda}\,B_{\lambda+\varepsilon}\,\le\, e^{\varepsilon \|h\|_1}\,Q_{\lambda+\varepsilon}\, B_{\lambda+\varepsilon}\,=\, e^{\varepsilon \|h\|_1}\,r(Q_{\lambda+\varepsilon})\, B_{\lambda+\varepsilon}
$$
and thus, invoking again \cite[Cor.12.4]{DanersKochMedina},
\bqn\label{22}
 e^{\varepsilon \|h\|_1}\, r(Q_{\lambda+\varepsilon})\, >\, r(Q_{\lambda})\ .
\eqn
Therefore, combining \eqref{21}, \eqref{22} and recalling that $r(Q_\lambda)$ is strictly decreasing in $\lambda$, we obtain
$$
e^{-\varepsilon \|h\|_1}\, r(Q_{\lambda})\,<\,r(Q_{\lambda+\varepsilon})\,<\,r(Q_{\lambda_j})\,<\,r(Q_{\lambda-\varepsilon})\,<\,e^{\varepsilon \|h\|_1}\, r(Q_{\lambda})\ .
$$
Letting $\varepsilon\rightarrow 0$ implies $\lim_{j\rightarrow\infty} r(Q_{\lambda_j})=r(Q_{\lambda})$, whence the continuity of the function \mbox{$\lambda\mapsto r(Q_\lambda)$}. Finally, the assumption that $h>0$ near $a=0$ together with \eqref{12} and \eqref{15} easily entails that $$0< r(Q_\lambda)\le \|Q_\lambda\|_{\mathcal{L}(E_\varsigma)}\rightarrow0\ ,\quad \lambda\rightarrow\infty\ ,$$ from which the assertion follows in view of \eqref{18}.
\end{proof}

\subsection{The nonlinear problem}

To investigate the nonlinear problem \eqref{5a}, \eqref{5b} we apply the theorem of Crandall-Rabinowitz \cite{CrandallRabinowitz}. Clearly, the solutions $(\lambda,u)=(\lambda_0+t,u)$ of \eqref{5a}, \eqref{5b} are the zeros of the function
\bqnn
F(t,u):=\left(\begin{array}{cc} \partial_a u+A(u)u+(\lambda_0+t)h u-g(\lambda_0+t,u)u\\
u(0)-\int_0^{a_m}b(u,a)u(a)\rd a\end{array}\right)\ .
\eqnn
Assumptions \eqref{7}, \eqref{8}, \eqref{9}, \eqref{10}, and \eqref{15} imply that
$$
F:(-\lambda_0,\infty)\times\Sigma\rightarrow \E_0\times E_\varsigma\quad\text{with} \quad F(t,0)=0\ ,\quad t>-\lambda_0\ .
$$
Moreover, it is easily seen that all partial derivatives $F_t$, $F_u$, and $F_{tu}$ exist and are continuous and that, since $g(\lambda,0)\equiv 0$, the Fr\'{e}chet derivatives at $(t,u)=(0,0)$ applied to $\varphi\in \E_1$ are given by
\bqn\label{30}
F_u(0,0)\varphi=\left(\begin{array}{cc} \partial_a \varphi+(A_0+\lambda_0 h)\varphi\\
\varphi(0)-\int_0^{a_m}b_0(a)\varphi(a)\rd a\end{array}\right)\ ,
\eqn
where $A_0:=A(0,\cdot)$, and 
\bqn\label{31}
F_{tu}(0,0)\varphi=\left(\begin{array}{cc} h\varphi\\
0\end{array}\right)\ .
\eqn
Recall that $B_{\lambda_0}\in\mathrm{int}(E_\varsigma^+)$ with $\mathrm{ker}(1-Q_{\lambda_0})=\mathrm{span}\{B_{\lambda_0}\}$. Then \eqref{11C} implies that $\Pi_{(\lambda_0,0)}(\cdot,0)B_{\lambda_0}$ belongs to $\E_1^+$. Moreover, for
\begin{align*}
&\ell_0(u):=\int_0^{a_m} b_0(a)u(a)\, \rd a\ ,\\
 & \big(K_0 f\big)(a):=\int_0^a \Pi_{(\lambda_0,0)}(a,\sigma) f(\sigma)\, \rd \sigma\ ,\quad f\in\E_0\ ,
\end{align*}
we deduce from \eqref{7}, \eqref{u1}, \eqref{u2}, and \eqref{15} that
\bqn\label{32}
\ell_0\in\mathcal{L}(\E_1,E_\varsigma)\ ,\quad K_0\in\mathcal{L}(\E_0,\E_1)\ .
\eqn
With these notations we can state the following result.

\begin{lem}\label{D}
$L:=F_u(0,0)\in\mathcal{L}(\E_1,\E_0\times E_\varsigma)$ is a Fredholm operator of index 0. In fact,
\begin{align*}
&\mathrm{ker}(L)=\mathrm{span}\{\Pi_{(\lambda_0,0)}(\cdot,0)B_{\lambda_0}\}\ ,\\
& \mathrm{rg}(L)=\big\{(\varphi,\psi)\in\E_0\times E_\varsigma\,;\, \psi+\ell_0(K_0\varphi)\in\mathrm{rg}(1-Q_{\lambda_0})\big\}
\end{align*}
are both closed and
$
\mathrm{dim}(\mathrm{ker}(L))=\mathrm{codim}(\mathrm{rg}(L))=1
$.
\end{lem}

\begin{proof}
This is a reformulation of \cite[Lem.2.1]{WalkerSIMA} using
\eqref{30}, \eqref{32}, \eqref{11C}, Lemma~\ref{H}, and Lemma~\ref{C}.
\end{proof}

This lemma also allows us to validate the transversality condition from \cite{CrandallRabinowitz}.

\begin{lem}\label{E}
We have $F_{tu}(0,0)\big(\Pi_{(\lambda_0,0)}(\cdot,0)B_{\lambda_0}\big)\not\in \mathrm{rg}(L)$.
\end{lem}

\begin{proof}
According to \eqref{31} and Lemma \ref{D} we have to check that
$$
z:=\ell_0\big(K_0(h\Pi_{(\lambda_0,0)}(\cdot,0)B_{\lambda_0})\big)\not\in\mathrm{rg}(1-Q_{\lambda_0})\ .
$$
Due to assumptions \eqref{9}, \eqref{15}, and properties of evolution operators, we compute
\bqnn\begin{split}
z&=\int_0^{a_m}b_0(a)\int_0^a\Pi_{(\lambda_0,0)}(a,\sigma)h(\sigma)\Pi_{(\lambda_0,0)}(\sigma,0) B_{\lambda_0}\ \rd\sigma\,\rd a\\
&=\int_0^{a_m}b_0(a)\left(\int_0^ah(\sigma)\,\rd\sigma\right)\,\Pi_{(\lambda_0,0)}(a,0) B_{\lambda_0}\,\rd a\ .
\end{split}
\eqnn
Thus $z\not=0$ due to \eqref{9}, \eqref{13}, and \eqref{15}. Using the commuting condition \eqref{14} we derive on interchanging the order of integration that
\bqnn\begin{split}
Q_{\lambda_0}z&=\int_0^{a_m}b_0(s)\,\Pi_{(\lambda_0,0)}(s,0)\int_0^{a_m}b_0(a)\left(\int_0^ah(\sigma)\,\rd\sigma\right)\, \Pi_{(\lambda_0,0)}(a,0) B_{\lambda_0}\ \rd a\,\rd s\\
&=\int_0^{a_m}b_0(a)\,\left(\int_0^ah(\sigma)\,\rd\sigma\right)\Pi_{(\lambda_0,0)}(a,0)\int_0^{a_m}b_0(s)\, \Pi_{(\lambda_0,0)}(s,0) B_{\lambda_0}\ \rd s\,\rd a\ .
\end{split}
\eqnn
Hence, simplifying the integral by recognizing $Q_{\lambda_0}B_{\lambda_0}=B_{\lambda_0}$ in the integrand and reversing the computations we obtain $Q_{\lambda_0}z=z$, that is, $z\in\mathrm{ker}(1-Q_{\lambda_0})$. But then $z\not\in\mathrm{rg}(1-Q_{\lambda_0})$ since $r(Q_{\lambda_0})=1$ is a simple eigenvalue of the compact operator $Q_{\lambda_0}$.
\end{proof}

Occurrence of local bifurcation in \eqref{5a}, \eqref{5b} is then a consequence of Lemma~\ref{D}, Lemma~\ref{E}, and \cite[Thm.1.7]{CrandallRabinowitz}.

\begin{thm}\label{F}
Suppose \eqref{6}-\eqref{11}, \eqref{11B}, \eqref{11C}, \eqref{13}-\eqref{15}, \eqref{18}. Further let $\lambda_0>0$ with $r(Q_{\lambda_0})=1$. Then $(\lambda_0,0)$ is a bifurcation point for \eqref{5a}, \eqref{5b}. More precisely, there are $\varepsilon_0>0$ and a unique branch $\{(\lambda(\ve),u(\ve))\,;\,\vert\ve\vert<\ve_0\}$ in $\R^+\times\E_1$ emanating from $(\lambda_0,0)$ with $u(\ve)\not\equiv 0$ if $\ve\not= 0$ of the form
\bqn\label{40}
u(\ve)=\ve\big(\Pi_{(\lambda_0,0)}(\cdot,0)B_{\lambda_0}+z(\ve)\big)\ ,\quad \vert\ve\vert<\ve_0\ .
\eqn
Both $\lambda:(-\ve_0,\ve_0)\rightarrow \R^+$ and $z:(-\ve_0,\ve_0)\rightarrow Z$ are continuous, where $\E_1=\mathrm{ker}(L)\oplus Z$ with an arbitrary complement $Z$. Moreover, $u(\ve)\in\E_1^+$ and $\gamma_0 u(\ve)\in\mathrm{int}(E_\varsigma^+)$ for $\ve\in (0,\ve_0)$.
\end{thm}

\begin{proof}
The existence of a nontrivial branch follows from Lemma~\ref{D} and Lemma~\ref{E} by applying \cite[Thm.1.7]{CrandallRabinowitz}. It remains to prove the positivity assertion for $\ve\in (0,\ve_0)$. Note that from \eqref{40} we have, for $\ve\in (0,\ve_0)$ with $\ve_0$ sufficiently small,
$$
\frac{1}{\ve}\gamma_0u(\ve)=B_{\lambda_0}+\gamma_0z(\ve)\in \mathrm{int}(E_\varsigma^+)
$$
since $z(\ve)\rightarrow 0$ in $\E_1\hookrightarrow BUC(J,E_\varsigma)$ and $B_{\lambda_0}\in\mathrm{int}(E_\varsigma^+)$. On the one hand, we derive $\gamma_0u(\ve)\in E_\varsigma^+$ and therefore, by positivity of the evolution operators assumed in \eqref{11},
$$
u(\ve)=\Pi_{(\lambda(\ve),u(\ve))}(\cdot,0)\gamma_0u(\ve)\in\E_1^+\ ,\quad \ve\in (0,\ve_0)\ .
$$
On the other hand, $\frac{1}{\ve}\gamma_0u(\ve)$ and thus also $\gamma_0u(\ve)$ are quasi-interior points of $E_\varsigma^+$, that is, $\langle f, \gamma_0u(\ve)\rangle_{E_\varsigma} >0$ for each $f\in E_\varsigma'\setminus\{0\}$ with $f\ge 0$. So \eqref{6} and \cite[Prop.A.2.10]{ClementEtAl} imply that $\gamma_0u(\ve)$ is an interior point of $E_\varsigma^+$.
\end{proof}

\begin{rem}\label{R1}
We proved a local bifurcation result for \eqref{5a}, \eqref{5b} under the assumption that $h\ge 0$. However, the statement of Theorem~\ref{F} still holds true if $h\le 0$. The only modification consists of replacing $h$ by $-h$ in assumption \eqref{9} so that the spectral radius $r(Q_\lambda)$ is strictly increasing in $\lambda$ (see Lemma~\ref{C}) and one thus has, in addition, to replace \eqref{18} by the assumption that $r(Q_0)<1$. 

Moreover, if $h<0$ one can even prove a global bifurcation result using the Rabinowitz alternative \cite{RBA} provided that the nonlinearities in the operator $A(u,a)$ are of ``lower order'', that is, if $A(u,a)$ is a sum of operators $A_0(a)+A_*(u,a)$, where $A_0(a)\in\mathcal{L}(E_1,E_0)$ and $A_*(u,a)\in\mathcal{L}(E_\theta,E_0)$ with $\theta\in [0,1)$. The approach is similar to \cite{WalkerJDE}. We also refer to \cite{DelgadoEtAl2} where the case $h\equiv -1$ is considered with linear diffusion.
\end{rem}

\section{Examples}\label{sec3}

Let $\Om\subset\R^N$, $N\ge 1$, be a bounded and smooth domain lying locally on one side of $\partial\Om$. Let the boundary $\partial\Om$ be the distinct union of two sets $\Gamma_0$ and $\Gamma_1$ both of which are open and closed in $\partial\Omega$. Let the maximal age be finite, i.e. let $a_m\in (0,\infty)$ and set $J:=[0,a_m]$. 

\subsection{A general example}\label{3.1}

Consider a second order differential operator of the form
\bqn\label{50a}
\Ac(U(x),a)w:=-\divv_x\big(D(U(x),a)\nabla_x w\big)+d\big(U(x),a\big)\cdot\nabla_x w\ ,
\eqn
%\eqn
where, for some $\rho>0$,
\bqn\label{50g}
D\in C^{5-,\rho}(\R\times J)\quad\text{with}\quad D(z,a)\ge d_0>0\ ,\quad z\in\R\ ,\quad a\in J\ ,
\eqn
and
\bqn
\begin{aligned}\label{50aa}
&d\in C^{4-,\rho}(\R\times J,\R^N)\quad \text{with}\quad d(0,\cdot)\equiv 0\ .
\end{aligned}
\eqn
For simplicity we refrain from an explicit dependence of $\Ac$ on $x\in \Om$.
Let 
\bqn\label{52}
\nu_0\in C^1(\Gamma_1)\ ,\qquad \nu_0(x)\ge 0\ ,\quad x\in\Gamma_1\ ,
\eqn
and let $\nu$ denote the outward unit normal to $\Gamma_1$. Let
\bqnn
\mathcal{B}(x)w:=\left\{\begin{array}{ll} w\ , & \text{on}\ \Gamma_0\ ,\\
 \frac{\partial}{\partial\nu}w+\nu_0(x) w\ , & \text{on}\ \Gamma_1\ .
\end{array}
\right.
\eqnn
Fix $p,q\in (1,\infty)$ with 
\bqn\label{50b}
\frac{2}{p}+\frac{N}{q}<1\ ,
\eqn
and let $E_0:=L_q:=L_q(\Om)$ be ordered by its positive cone of functions that are nonnegative almost everywhere. Observe that 
$$
E_1:=\Wqb^2:=\Wqb^2(\Om):=\big\{u\in W_q^2\,;\, \mathcal{B}u=0\big\}\dhr L_q=E_0 \ ,
$$ 
where $W_q^2(\Om)$ is the usual Sobolev space of order 2 over $L_q(\Om)$. Also note that, up to equivalent norms, the real interpolation spaces between $E_0$ and $E_1$ are subspaces of the Besov spaces $B_{q,p}^{2\xi}:=B_{q,p}^{2\xi}(\Om)$, that is,
\bqnn
E_\xi:=\big(L_q,W_{q,\mathcal{B}}^2\big)_{\xi,p}\,\dot{=}\,\Bqb^{2\xi}:=\left\{\begin{array}{ll} B_{q,p}^{2\xi}\ ,\quad &0<{2\xi}<1/q\ ,\\
\big\{w\in B_{q,p}^{2\xi}\,;\, u\vert_{\Gamma_0}=0\big\}\ ,& 1/q<{2\xi}<1+1/q\ , 2\xi\not= 1\ ,\\
\big\{w\in B_{q,p}^{2\xi}\,;\, \mathcal{B}u=0\big\}\ ,& 1+1/q<2\xi<2\ ,
\end{array}
\right.
\eqnn
(see e.g. \cite{Triebel}).
In particular, due to \eqref{50b} we have $E_{\varsigma}\,\dot{=}\,\Bqb^{2-2/p}\hookrightarrow C^{1+\epsilon}(\bar{\Om})$ for $\varsigma=1-1/p$ and some $\epsilon>0$. So $\mathrm{int}(E_\varsigma^+)\ne\emptyset$ yielding \eqref{6}. Fix any $\kappa\in (2-2/p,2)\setminus\{1\}$ and set $F:=\Bqb^{\kappa}$ with order induced by the cone of $L_q$. Then pointwise multiplication
$$
\Bqb^{\kappa}\cdot W_{q,\mathcal{B}}^2\hookrightarrow \Bqb^{2(1-1/p)}\doteq E_\varsigma
$$
is continuous according to \eqref{50b} and \cite[Thm.4.1]{AmannMultiplication}. Thus \eqref{7} holds. Let 
$$
\E_1:=L_p(J,\Wqb^2)\cap W_p^1(J,L_q)\quad\text{and}\quad \E_0:=L_p(J,L_q)
$$
and note that 
\bqnn
U:=\int_0^{a_m}u(a)\rd a\in E_1=\Wqb^2\ ,\quad u\in\E_1\ .
\eqnn
Suppose that
\bqn\label{60}
\mu\in C^{4-,\rho}(\R\times J)\ ,\qquad \mu\ge 0\ ,\qquad \mu(0,a)>0\ \text{for}\ a\ \text{near}\ 0\ .
\eqn
Set $h(a):=\mu(0,a)$ and $g(\lambda,u)(a):=\lambda(\mu(U,a)-h(a))$ for $\lambda\in\R$, $u\in\E_1$, and $a\in J$. Then \eqref{60} and Proposition~\ref{dix} from the appendix ensure \eqref{8} and \eqref{9}.
Further suppose that
\bqn\label{61}
\bar{b}\in C^{4-,0}(\R\times J)\ ,\qquad \bar{b}\ge 0\ ,\qquad \bar{b}(0,\cdot)\not\equiv 0\ ,
\eqn
and set $b(u,a):=\bar{b}(U,a)$ for $u\in\E_1$, $a\in J$.
Then \eqref{61} and Proposition~\ref{dix} ensure \eqref{15}. 
Define
$$
A(u,a)w:=\Ac(U,a)w\ ,\quad w\in E_1\ ,\quad u\in \E_1\ .
$$
Proposition~\ref{dix}, \eqref{50g}, and \eqref{50aa} entail that the superposition operators induced by $D$, $\partial_1D$, and $d$ (again labeled $D$, $\partial_1D$, and $d$) satisfy $D,\partial_1D, d\in C^1(\Wqb^2,L_\infty(J,C^{1+\epsilon}(\Om)))$. This yields $$A\in C^1(\Wqb^2,L_\infty(J,\mathcal{L}(\Wqb^2,L_q)))\ ,$$
 whence \eqref{10}. If $u\in \E_1$ and $\lambda\ge 0$ are fixed, then $A(u,\cdot)+\lambda\mu(U,\cdot)\in C^\rho(J,\mathcal{H}(\Wqb^2,L_q))$ from which we conclude \eqref{11} and \eqref{11B} due to \cite[I.Cor.1.3.2, II.Cor.4.4.2]{LQPP} and the compactness of $J$. Noticing that $A(0,a)=-D(0,a)\Delta_x$ by \eqref{50aa} it follows from \cite[Sect.7,Thm.11.1]{AmannIsrael} that for $a\in J$ and $\lambda>0$ fixed, $-A(0,a)-\lambda h(a)$ is resolvent positive, generates a contraction semigroup of negative type on each $L_r(\Om)$, $r\in (1,\infty)$, and is self-adjoint on $L_2(\Om)$. Hence $A(0,\cdot)+\lambda h$ possesses maximal $L_p$-regularity on $J$ according to \cite[III.Ex.4.7.3,III.Thm.4.10.8]{LQPP}, whence \eqref{11C}. Moreover, since 
$$
\Pi_0(a,\sigma)=e^{\int_\sigma^aD(0,r)\rd r \Delta_x}\ ,\quad 0\le \sigma\le a\le a_m\ ,
$$
where $\{e^{a\Delta_x }; a\ge 0\}$ is the semigroup associated with $(-\Delta_x,\B)$, condition \eqref{13} follows from the maximum principle and \eqref{14} is obvious. Finally, let $\sigma_1$ be the first eigenvalue of $(-\Delta_x,\B)$ and let $\varphi_1\in \Wqb^2$ be a corresponding positive eigenfunction (e.g. see \cite{AmannIsrael}). 
Then
$$
e^{\int_0^aD(0,r)\rd r \Delta_x}\,\varphi_1\,=\,e^{-\sigma_1\int_0^a D(0,r)\rd r}\varphi_1\ ,\quad 0\le a\le a_m\ ,
$$
and thus
$$
Q_\lambda\,\varphi_1\,=\,\int_0^{a_m}b(0,a)\, e^{-\lambda\int_0^a\mu(0,r)\rd r}\, e^{\int_0^aD(0,r)\rd r \Delta_x}\,\varphi_1\,\rd a\, =\, k(\lambda)\,\varphi_1\ ,
$$
where
$$
k(\lambda)\,:=\,\int_0^{a_m}b(0,a)\, e^{-\lambda\int_0^a\mu(0,r)\rd r}\, e^{-\sigma_1\int_0^a D(0,r)\rd r}\,\rd a\ .
$$
Since the spectral radius $r(Q_\lambda)$ is the only eigenvalue with positive eigenfunction for the strongly positive compact operator $Q_\lambda\in \mathcal{K}(\Bqb^{2-2/p})$ by the Krein-Rutman theorem, we have $r(Q_\lambda)=k(\lambda)$. To satisfy \eqref{18} we assume that
 \bqn\label{333}
 k(0)\,=\, \int_0^{a_m}b(0,a)\, e^{-\sigma_1\int_0^a D(0,r)\rd r}\,\rd a>1
 \eqn
and then choose $\lambda_0>0$ such that $k(\lambda_0)=1$.\\

Summarizing what we have just shown and referring to Theorem~\ref{F} we can state:

\begin{prop}\label{G}
Suppose \eqref{50a}-\eqref{333}. Then $(\lambda_0,0)$ with $k(\lambda_0)=1$ is a bifurcation point for the problem
\begin{align*}
&\partial_a u+\Ac(U(x),a)u+\lambda\mu(U(x),a)u=0\ ,&a\in J\, ,\ x\in\Om\ ,\\
&u(0,x)=\int_0^{a_m} b\big(U(x),a)u(a)\,\rd a\ ,&  x\in\Om\ ,\\
&\B u(a,x)=0\ ,&a>0\, ,\ x\in\partial\Om\ ,\\
&U(x)=\int_0^{a_m}u(a,x)\,\rd a\ ,&  x\in\Om\ .
\end{align*}
There are $\varepsilon_0>0$ and a unique branch $\{(\lambda(\ve),u(\ve))\,;\,\vert\ve\vert<\ve_0\}$ of solutions emanating from $(\lambda_0,0)$ with $$u(\ve)\in L_p(J,\Wqb^2)\cap W_p^1(J,L_q),\quad u(\ve)\not\equiv 0\ \text{if}\ \ve\not= 0\ ,$$ of the form
\bqn\label{400}
u(\ve)=\ve\big(\,e^{-\lambda_0\int_0^a\mu(0,r)\rd r}\,e^{-\sigma_1\int_0^a D(0,r)\rd r}\varphi_1+z(\ve)\big)\ ,\quad \vert\ve\vert<\ve_0\ .
\eqn
Both $\lambda:(-\ve_0,\ve_0)\rightarrow \R^+$ and $z:(-\ve_0,\ve_0)\rightarrow L_p(J,\Wqb^2)\cap W_p^1(J,L_q)$ are continuous. Moreover, $u(\ve)(a,x)\ge 0$ for $\ve\in (0,\ve_0)$ and $(a,x)\in J\times\Om$.
\end{prop}

\begin{rem}
A local dependence of the data on $u$ with respect to age is also possible. For example, one may apply Theorem~\ref{F} for diffusion terms of the form $\divv_x(D(u(a,x))\nabla_ u)$ as well (see \cite[Ex.3.1]{WalkerSIMA} for details). Moreover, the functions $D$ and $d$ in \eqref{50a} may also depend on $x\in\Om$ provided the dependence is sufficiently smooth. In this case one needs Remark~\ref{C1} d) to verify \eqref{10}.

%(b) That $J$ is compact is not needed in the previous example. If $J=\R^+$, then $$\E_1\hookrightarrow W_p^\theta %(J,E_{1-\theta})\hookrightarrow L_1(J,E_{1-\theta})\ ,\quad \theta\in (0,1)\ ,$$ by interpolation and Sobolev %embedding. Thus $U\in E_{1-\theta}$ for $u\in\E_1=\Bqb^{2-2\theta}$ and Proposition~\ref{dix} can again be used for the %differentiability assertions. In this case more assumptions have to be imposed on the data.
\end{rem}

\subsection{Proof of Theorem \ref{P}}

Applying Proposition \ref{G} to the problem
\begin{align}
&\partial_a u=\divv_x\big(D(U(x),a)\nabla_xu\big)-\lambda\mu(U(x),a)u\ ,& a\in J\, ,\ x\in\Om\ ,\label{2aa}\\
&u(0,x)=\int_0^{a_m} b\big(U(x),a)u(a)\,\rd a\ ,&  x\in\Om\ ,\label{2bb}\\
&\delta u(a,x)+(1-\delta)\partial_\nu u(a,x)=0\ ,& a\in J\, ,\ x\in\partial\Om\ .\label{2cc}
\end{align}
considered in the introduction, we obtain under the assumptions of Theorem~\ref{P} a branch of nontrivial solutions 
$$
\{(\lambda(\ve),u(\ve))\,;\,\vert\ve\vert<\ve_0\}\quad\text{in}\quad \R^+\times \big( C(J,C(\bar{\Om}))\cap C^{1}(\dot{J},C(\bar{\Om}))\cap C(\dot{J},C^2(\Om))\big)\ ,
$$ 
where the regularity of $u(\ve)$ is due to standard parabolic regularity theory (e.g. see \cite[Thm.9.2]{AmannReactDiffII}), where $\dot{J}:=J\setminus\{0\}$.

Let now $\delta=0$ in \eqref{2cc} and assume \eqref{4}. Then $\sigma_1=0$ and, for any nonnegative solution $(\lambda,u)$ to \eqref{2aa}-\eqref{2cc}, we have
$$
\dfrac{\rd}{\rd a}\int_\Om u(a,x)\,\rd x=-\lambda\int_\Om \mu(U(x),a) u(a,x)\,\rd x\le -\lambda\mu(0,a)\int_\Om u(a,x)\,\rd x\ ,
$$
whence
$$
z(a)\le z(0)e^{-\lambda\int_0^a \mu(0,r)\rd r}\quad\text{for}\quad z(a):=\int_\Om u(a,x)\,\rd x\ .
$$
Moreover, by \eqref{2bb} and \eqref{4},
$$ z(0)=\int_\Om\int_0^{a_m}b(U(x),a) u(a,x)\,\rd a\rd x\le \int_0^{a_m}b(0,a) z(a)\,\rd a\le z(0)\int_0^{a_m}b(0,a)e^{-\lambda\int_0^a \mu(0,r)\rd r}\,\rd a\ ,
$$
and thus
$$
k(\lambda)=\int_0^{a_m}b(0,a)e^{-\lambda\int_0^a \mu(0,r)\rd r}\,\rd a\ge 1\ .
$$
Since $k(\lambda)$ is strictly decreasing in $\lambda$ and $k(\lambda_0)=1$, we conclude $\lambda\le\lambda_0$, that is, subcritical bifurcation occurs in \eqref{2aa}-\eqref{2cc} in this case. This proves Theorem~\ref{P}.

\subsection{Subcritical bifurcation for Dirichlet boundary conditions}

We consider an example involving Di\-richlet boundary conditions. More precisely, let us consider
\begin{align}
&\partial_a u=D(U)\Delta_xu-\lambda\mu(U(x),a)u\ ,& a\in J\, ,\ x\in\Om\ ,\label{500}\\
&u(0,x)=\int_0^{a_m} b\big(U(x),a)u(a)\,\rd a\ ,&  x\in\Om\ ,\label{510}\\
&u(a,x)=0\ ,& a\in J\, ,\ x\in\partial\Om\ .\label{520}
\end{align}
Note that the diffusion coefficients are independent of $a\in J$. Suppose that $D\in C^1(\Wqb^2(\Om),(0,\infty))$. If \eqref{4} still holds, then we easily derive for any positive solution $(\lambda,u)$ of \eqref{500}-\eqref{520} analogously as above that
$$
z'(a)+\sigma_1 D(U)z(a)\le -\lambda \mu(0,a) z(a)\quad\text{for}\quad z(a):=\int_\Om \varphi_1(x) u(a,x)\,\rd x\ ,
$$
where $\varphi_1$ is a positive eigenfunction to the principal eigenvalue $\sigma_1>0$ of $-\Delta_x$ subject to Dirichlet boundary conditions on $\partial\Om$. Thus, if in addition $D(U)\ge D(0)$ for $0\le U\in\Wqb^2(\Om)$, then we deduce again that
$$
z(0)\le \int_0^{a_m}b(0,a)\,e^{-\sigma_1 D(0)a}\, e^{-\lambda\int_0^a\mu(0,r)\rd r}\, \rd a\, z(0)\,=\,k(\lambda)\, z(0)\ .
$$
Hence $\lambda\le\lambda_0$ and thus subcritical bifurcation occurs also in this case.

Observe that one may replace the diffusion term $D(U)\Delta_xu$ in \eqref{500} by $\divv_x(D(U(x))\nabla_x u)$ depending locally with respect to $x$ on $U$ and derive the same conclusion of subcritical bifurcation provided that \mbox{$\sigma_1(U)\ge \sigma_1(0)$} for $0\le U\in \Wqb^2(\Om)$, where $\sigma_1(U)$ is the first eigenvalue of $w\mapsto -\divv_x(D(U(x))\nabla_x w)$ and using a corresponding positive eigenfunction $\varphi_1=\varphi_1(U)$ in the definition of $z$.

\subsection{An example with Holling-Tanner type nonlinearities}

As noted in Remark~\ref{R1} one can also allow for $h<0$ in \eqref{5a}. We conclude with an example, which has been investigated in \cite{DelgadoEtAl2} in the case of linear diffusion:
\begin{align}
&\partial_a u+\Ac(U(x),a)u+\mu(U(x),a)u=\lambda u\pm\dfrac{u}{1+u}\ ,& a\in J\, ,\ x\in\Om\ ,\label{60a}\\
&u(0,x)=\int_0^{a_m} b\big(U(x),a)u(a)\,\rd a\ ,&  x\in\Om\ ,\label{60b}\\
&\B u(a,x)=0\ ,& a>0\, ,\ x\in\partial\Om\ ,\label{60c}\\
&U(x)=\int_0^{a_m}u(a,x)\,\rd a\ ,&  x\in\Om\ ,\label{60d}
\end{align}
with $\Ac$ and $\B$ as in Subsection~\ref{3.1}. We impose the same assumptions \eqref{50a}-\eqref{61} as there, where we take $q=p$ for simplicity. The strict positivity of $\mu(0,a)$ in \eqref{60} is not needed here. We also use the same spaces as in Subsection~\ref{3.1}:
\begin{align*}
&E_1:=\Wpb^2:=\Wpb^2(\Om):=\big\{u\in W_p^2\,;\, \mathcal{B}u=0\big\}\dhr L_p=:E_0 \ ,\quad E_\varsigma\doteq \Wpb^{2-2/p} \ ,\\
&\E_1:=L_p(J,\Wpb^2)\cap W_p^1(J,L_p)\quad\text{and}\quad \E_0:=L_p(J,L_p)\ .
\end{align*}
Noticing that the ``Holling-Tanner-type'' nonlinearity can be written in the form
$$
\dfrac{u}{1+u}=u-\dfrac{u^2}{1+u}\ ,
$$
problem \eqref{60a}-\eqref{60d} fits in the abstract form \eqref{5a}, \eqref{5b} by setting
$$ 
A(u,a):=\Ac(U,a)+\mu(U,a)\mp 1\in\mathcal{L}(\Wpb^2,L_p)\ ,\quad h(a):=-1\ ,\quad g(u):=\mp\frac{u}{1+u}\ .
$$
Let $V:=(-1,1)$. Then \eqref{50b} (with $p=q$), Remark~\ref{C1} b) from the appendix, and \cite[VII.Thm.6.4]{AmannEscherII} ensure that the superposition operator of $g$ (still denoted by $g$) belongs to $C^1(C(J,V_p),C(J,\Wpb^{2-2/p}))$, where we define \mbox{$V_p:=\Wpb^{2-2/p}\cap C(\bar{\Om},V)$}. Also note that $g(0)=0$. Recalling that $\E_1\hookrightarrow C(J,E_\varsigma)$, this implies $g\in C^1(\Sigma, C(J,F))$ for $\Sigma:=\mathbb{B}_{\E_1}(0,R)$ with $R>0$ sufficiently small and $F:=E_\varsigma$. To satisfy $r(Q_0)<1$ (see Remark~\ref{R1}), we assume that
\bqn\label{100}
k(0)<1\ 
\eqn
for
$$
k(\lambda):=\int_0^{a_m}b(0,a)\, e^{(\lambda\mp 1)a}\, e^{-\int_0^a\mu(0,r)\rd r}\,e^{-\sigma_1\int_0^a D(0,r)\rd r}\, \rd a\ ,
$$
where $\sigma_1$ is the first eigenvalue of $(-\Delta_x,\B)$. Let $\varphi_1\in \Wpb^2$ be a corresponding positive eigenfunction. As in Subsection~\ref{3.1} we have $Q_\lambda\varphi=k(\lambda)\varphi_1$ with
$$
Q_\lambda:=\int_0^{a_m}b(0,a)\, e^{(\lambda\mp 1)a}\, e^{-\int_0^a\mu(0,r)\rd r}\,e^{\int_0^a D(0,r)\rd r\Delta_x}\, \rd a\ ,
$$
whence $r(Q_\lambda)=k(\lambda)$, in particular, $r(Q_0)<1$ in view of \eqref{100}. Thus we may invoke Theorem~\ref{F} and Remark~\ref{R1} to conclude the existence of a branch of nontrivial solutions to \eqref{60a}-\eqref{60d} emanating from the critical point $(\lambda_0,0)$, where $\lambda_0>0$ with $k(\lambda_0)=1$.

This local bifurcation result generalizes the (global) one of \cite{DelgadoEtAl2} in that nonlinear diffusion and nonlinear death and birth rates may be considered. However, we shall point out that in \cite{DelgadoEtAl2} a death rate depending on local position is considered (what can be considered in the present situation as well but requires some additional effort).

\section{Appendix}\label{appendix}

We prove a result on the differentiability of superposition operators in Sobolev-Slobodeckii spaces that is used in the previous examples but might be of interest in other applications as well. The proof is similar to \cite[Lem.2.7]{WalkerEJAM} or \cite[Prop.15.4]{AmannAnnali88}, where continuity properties are derived.

To set the stage let $\Om$ be an open and bounded subset of $\R^n$, let $V$ be an open neighborhood of 0 in $\R^k$, and let $I$ be a compact interval in $\R$. Given a function
$f:I\times V\rightarrow \R$ define the {\it superposition operator $F$ of $f$} by
$$F[u](a)(x):=f(a,u(x))\quad\text{for}\ u:\Om\rightarrow V\ \text{and}\ a\in I\ ,\ x\in \Om\ .
$$
We write $f\in C^{0,k-}(I\times V)$ provided that $\partial_2^{k-1}f\in C(I\times V)$ is Lipschitz continuous in $x\in V$ uniformly with respect to $a\in I$.

Recall the definition of the norm in $W_q^\xi(\Om,\R^k)$ (e.g. see \cite{Triebel}): if $q\in (1,\infty)$ and $\xi\in (0,1)$, then
$$
\|u\|_{W_q^\xi(\Om,\R^k)}^q=\|u\|_{L_q(\Om,\R^k)}^q +\int_{\Om\times\Om}\dfrac{\vert u(x)-u(y)\vert^q}{\vert x-y\vert^{n+\xi q}}\,\rd (x,y)
$$
and if $\xi\in (1,2)$, then
$$
\|u\|_{W_q^\xi(\Om,\R^k)}^q=\|u\|_{L_q(\Om,\R^k)}^q+\sum_{j=1}^n\|\partial_ju\|_{W_q^{\xi-1}(\Om,\R^k)}^q\ .
$$

\begin{prop}\label{dix}
Let $q\in (n,\infty)$, $\xi\in(n/q,2)$, and
$\eta\in(0,\xi)$. Further, let $V_\xi:=W_q^\xi(\Om,\R^k)\cap C(\bar{\Om},V)$ be equipped with the $W_q^\xi$-topology. Then $F\in
C^{2-}\big(V_\xi, L_\infty(I,W_q^{\eta}(\Om))\big)$ provided that $f\in C^{0,4-}(I\times V)$. The Fr\'{e}chet derivative $DF[u]$ at $u\in V_\xi$ is given by $$\big(DF[u]h\big)(a)(x)=\partial_2f(a,u(x))h(x)\ ,\quad a\in I\ ,\ x\in\Om\ ,\   h\in W_q^\xi(\Om,\R^k)\ .$$
\end{prop}

\begin{proof}
We may assume $k=1$.\\

\noindent (i) First, let $\xi\in(n/q,1)$ and note that $W_q^\xi(\Om)\hookrightarrow
C(\bar{\Om})$. Fix $u\in V_\xi$ and choose an open neighborhood $R$ of $u(\bar{\Om})$ in $V$ such that its closure $\bar{R}$ is compact and contained in $V$. Then, since $f\in C^{0,3-}(I\times V)$, there is $c_0(R)>0$ with 
\begin{align*}
&\vert \partial_2^2f(a,r)\vert\le c_0(R)\ ,\quad r\in\bar{R}\ ,\quad a\in I\ ,\\
& \vert \partial_2f(a,r)-\partial_2f(a,s)\vert+ \vert \partial_2^2f(a,r)-\partial_2^2f(a,s)\vert\,\le\, c_0(R)\, \vert
r-s\vert\ ,\quad r,s\in\bar{R}\ ,\quad a\in I\ .
\end{align*}
In the following we suppress the (fixed) variable $a\in I$ in $f$ and its derivatives for the sake of readability and we set $f':=\partial_2 f$. Let $h\in W_q^\xi(\Om)$ with $\|h\|_{W_q^\xi(\Om)}$ sufficiently small so that $u(\bar{\Om})+h(\bar{\Om})\subset\bar{R}$. Then $$(F'[u]h)(a)(x):=\partial_2f(a,u(x))h(x)=f'(u(x))h(x)$$ by convention. The mean value theorem implies for $x,y\in \Om$:
 \begin{equation*}
    \begin{split}
    \big \vert & F[u+h](a)(x)-F[u](a)(x)-(F'[u]h)(a)(x)-\big[F[u+h](a)(y)-F[u](a)(y)-(F'[u]h)(a)(y) \big]\big\vert\\
    &\le\, \left\vert \int_0^1 \big[f'\big(u(x)+\tau h(x))-f'(u(x))\big]\,\rd \tau\,\big(h(x)-h(y)\big)\right\vert\\
    &\quad+\left\vert \int_0^1\int_0^1 f''\big(u(x)+\sigma\tau h(x))\,\rd \sigma\, h(x)\,\rd \tau\, h(y)
     - \int_0^1\int_0^1 f''\big(u(y)+\sigma\tau h(y)\big)\,\rd \sigma\, h(y)\,\rd \tau\, h(y)\right\vert\\
    &\le \, c_0(R)\,\| h\|_\infty\, \big\vert h(x)-h(y)\big\vert + \left\vert \int_0^1\int_0^1 f''\big(u(x)+\sigma\tau h(x)\big)\,\rd\sigma\rd\tau\,\big(h(x)-h(y)\big)\, h(y)\right\vert\\
    &\qquad +\left\vert \int_0^1\int_0^1 \big[f''\big(u(x)+\sigma\tau h(x)\big)\,-\,f''\big(u(y)+\sigma\tau h(y)\big)\big]\,\rd \sigma\rd\tau\, h(y)\, h(y)\right\vert\\
    &\le\, c(R)\,\| h\|_\infty\, \big\vert h(x)-h(y)\big\vert\,+ \,c_0(R)\,\| h\|_\infty^2\, \big[\vert u(x)-u(y)\vert+\vert h(x)-h(y)\vert\big]\ .
    \end{split}
    \end{equation*}
Therefore, we obtain by definition of the norm in the Sobolev-Slobodeckii space $W_q^\xi(\Om)$ that
    \begin{equation*}
    \begin{split}
    \big\|F[u+h]&(a)-F[u](a)-(F'[u]h)(a)\|_{W_q^\xi(\Om)}^q \\
    %&\le\,\big\|F[u+h](a)-F[u](a)-(F'[u]h)(a)\|_{L_q(\Om)}^q\\
    %&\quad +  \int_{\Om\times\Om} \Big\vert F[u+h](a)(x)-F[u](a)(x)-(F'[u]h)(a)(x)\\
    %&\qquad\qquad\qquad-\big[F[u+h](a)(y)-F[u](a)(y)-(F'[u]h)(a)(y) \big]\Big\vert^q
    %\, \dfrac{\rd (x,y)}{ \vert
    %x-y\vert^{n+\xi q}}\\
     &\le \, c(R)\,\| h\|_\infty^{2q} \,+\, c(R)\,\| h\|_\infty^q\,\| h\|_{W_q^\xi(\Om)}^q\,+\,c(R)\, \| h\|_\infty^{2q}\,\left\{\| u\|_{W_q^\xi(\Om)}^q \,+\, \|h\|_{W_q^\xi(\Om)}^q\right\}\ .
    \end{split}
    \end{equation*}
Recalling the embedding $W_q^\xi(\Om)\hookrightarrow C(\bar{\Om})$ we thus deduce that    
$$
\big\|F[u+h]-F[u]-F'[u]h\|_{L_\infty(I,W_q^\xi(\Om))}\, =\, o\big(\|h \|_{W_q^\xi(\Om)}\big)\ ,\quad (h\rightarrow 0)\ ,
$$    
whence $F:V_\xi\rightarrow L_\infty(I,W_q^\xi(\Om)\big)$ is Fr\'{e}chet differentiable at $u\in V_\xi$ with derivative $DF[u]h=F'[u]h$ for $h\in W_q^\xi(\Om)$. Moreover, since pointwise multiplication $W_q^\xi(\Om)\times W_q^\xi(\Om)\rightarrow W_q^\xi(\Om)$ is continuous due to \cite[Thm.4.1]{AmannMultiplication} and $\xi>n/q$, we have for $v\in W_q^\xi(\Om)$ with sufficiently small norm that (writing here and in the following e.g. $f'(u)$ for the superposition operator at $u$ induced by $f'$)
\bqnn
\begin{split}
\|DF[u+v]-DF[u]&\|_{\mathcal{L}(W_q^\xi(\Om),L_\infty(I,W_q^{\xi}(\Om)))}\, =\,\sup_{h\in W_q^\xi(\Om)}\,\dfrac{\|f'(u+v)h-f'(u)h\|_{L_\infty(I,W_q^{\xi}(\Om))}}{\|h\|_{W_q^\xi(\Om)}}\\
&\le c\, \|f'(u+v)-f'(u)\|_{L_\infty(I,W_q^\xi(\Om))}\le\, c(R)\,\|v\|_{W_q^\xi(\Om)}\ ,
\end{split}
\eqnn
where the last inequality can be shown similarly as above (or also follows from \cite[Lem.2.7]{WalkerEJAM}). This implies $F\in C^{2-}\big(V_\xi, L_\infty(I,W_q^{\xi}(\Om))\big)$.\\

\noindent (ii) Now let $\xi\in (1,2)$ and $\eta\in
(1,\xi)$. Choose $\tau\in (n/q ,1)$ with
$\tau>\eta-1$. Then pointwise multiplication $W_q^\tau(\Om)
\times W_q^{\xi-1}(\Om)\rightarrow W_q^{\eta-1}(\Om)$ is
continuous, see \cite[Thm.4.1]{AmannMultiplication}. Therefore, taking $\xi=\tau$ in (i) and using the chain rule we obtain for $u\in V_\xi$ and $h\in W_q^\xi(\Om)$ with $\|h\|_{W_q^\xi(\Om)}$ sufficiently small that
     \begin{equation*}
    \begin{split}
    \big\|F[u+h]-F[&u]-F'[u]h\|_{L_\infty(I,W_q^{\eta}(\Om))}\\
    &\le\, \|F[u+h]-F[u]-F'[u]h\|_{L_\infty(I,L_q(\Om))}\\
    &\quad+\,\sum_{j=1}^n \big\|\big(f'(u+h)-f'(u)-f''(u)h\big)\, \partial_j u\big\|_{W_q^{\eta-1}(\Om)} \\
    &\quad  + \sum_{j=1}^n \big\|\big(f'(u+h)-f'(u)\big)\,\partial_j h \big\|_{W_q^{\eta-1}(\Om)} \\
      &\le \,  o\big(\|h \|_{W_q^\xi(\Om)}\big)\, +\, c\, \big\|f'(u+h)-f'(u)-f''(u)h\big\|_{W_q^{\tau}(\Om)} \, \| 					u\big\|_{W_q^{\xi}(\Om)} \\
      &\quad +\,c\,\| f'(u+h)-f'(u)\|_{W_q^{\tau}(\Om)}\,\|h\|_{W_q^{\xi}(\Om)}\\
      &\le\, o\big(\|h \|_{W_q^\xi(\Om)}\big)\, +\, o\big(\|h \|_{W_q^\tau(\Om)}\big)\,\| u\big\|_{W_q^{\xi}(\Om)}\, +\, 		c\, \|h\|_{W_q^\tau(\Om)}\,\|h\|_{W_q^\xi(\Om)}\\
      &=  o\big(\|h \|_{W_q^\xi(\Om)}\big)\ ,\quad (h\rightarrow 0)
	\end{split}
    \end{equation*}
the last equality being due to the embedding $W_q^\xi(\Om)\hookrightarrow W_q^\tau(\Om)$. This shows that
$F:V_\xi\rightarrow L_\infty(I,W_q^{\eta}(\Om)\big)$ is Fr\'{e}chet differentiable at $u\in V_\xi$ with derivative $DF[u]h=F'[u]h$ for $h\in W_q^\xi(\Om)$. 

Finally, to prove the Lipschitz continuity of $DF$ choose $\bar{\xi}\in (\eta,\xi)$ and note that pointwise multiplication $W_q^\xi(\Om)\times W_q^{\bar{\xi}}(\Om)\rightarrow W_q^\eta(\Om)$ is continuous due to \cite[Thm.4.1]{AmannMultiplication}. Hence it follows for $v\in W_q^\xi(\Om)$ with sufficiently small norm that
\bqnn
\begin{split}
\|DF[u+v]-DF[u]&\|_{\mathcal{L}(W_q^\xi(\Om),L_\infty(I,W_q^{\eta}(\Om)))}\,=\, \sup_{h\in W_q^\xi(\Om)}\,\dfrac{\|f'(u+v)h-f'(u)h\|_{L_\infty(I,W_q^{\eta}(\Om))}}{\|h\|_{W_q^\xi(\Om)}}\\
&\le c\, \|f'(u+v)-f'(u)\|_{L_\infty(I,W_q^{\bar{\xi}}(\Om))}\le\, c(R)\,\|v\|_{W_q^\xi(\Om)}\ ,
\end{split}
\eqnn
where the last inequality stems from \cite[Lem.2.7]{WalkerEJAM}. So $F\in C^{2-}\big(V_\xi, L_\infty(I,W_q^{\eta}(\Om))\big)$.
The case $\xi=1$ is obvious.
\end{proof}

\begin{rems}\label{C1}
(a) If $\xi\in (n/q,1)$ and $f\in C^{0,3-}(I\times V)$, then $F\in
C^{2-}\big(V_\xi, L_\infty(I,W_q^{\xi}(\Om))\big)$.

\begin{proof}
See part (i) of the proof of Proposition~\ref{dix}.
\end{proof}

(b) If $\xi\in (1+n/q,2)$ and $f\in C^{0,4-}(I\times V)$, then $F\in
C^{2-}\big(V_\xi, L_\infty(I,W_q^{\xi}(\Om))\big)$.

\begin{proof}
This follows exactly as in part (ii) of the proof of Proposition~\ref{dix} by observing that pointwise multiplication $W_q^\tau(\Om)
\times W_q^{\xi-1}(\Om)\rightarrow W_q^{\xi-1}(\Om)$ for $\tau\in (n/q ,1)$ with
$\tau>\xi-1$ is
continuous according to \cite[Thm.4.1]{AmannMultiplication}.
\end{proof}

 %If $1\le k<\eta<\xi<k+1$ and $f\in C^{0,(k+3)-}(I\times V)$, then \mbox{$F\in
%C^{2-}\big(V_\xi, L_\infty(I,W_q^{\eta}(\Om))\big)$}.

%\begin{proof}
%This can be shown inductively as in part (ii) of the proof of Proposition~\ref{dix}.
%\end{proof}

(c) For simplicity we refrain form taking into account an explicit dependence of $f$ on $x\in\Om$, that is, we do not consider functions $f:I\times\bar{\Om}\times V\rightarrow\R$ in Proposition~\ref{dix}. Such a dependence can be included provided that $f$ (and its derivatives) are H\"older continuous with respect to $x$, see \cite[Prop.15.4, Prop.15.6]{AmannAnnali88}.

\end{rems}

\section*{Acknowledgement}
Part of this paper was written while visiting the University of Strathclyde in Glasgow. I would like to thank for the kind hospitality and support.


\begin{thebibliography}{99}


\bibitem{AmannIsrael}
H. Amann. \textit{Dual semigroups and second order linear elliptic boundary value problems.} Israel J. Math. {\bf 45} (1983), 225-254.

\bibitem{AmannAnnali88}
H. Amann. \textit{Existence and regularity for semilinear parabolic evolution equations.} Annali Scu. norm. sup. Pisa {\bf 11} (1988), 593-676.

\bibitem{AmannReactDiffII}
H. Amann. \textit{Dynamic theory of quasilinear parabolic equations. II. Reaction-diffusion systems.} Differential Integral Equations {\bf 3} (1990), 13-75.

\bibitem{AmannMultiplication}
H. Amann. \textit{ Multiplication in Sobolev and Besov spaces},
In Nonlinear analysis. A tribute in honour of Giovanni
Prodi. 27-57, Quaderni, Scuola Norm. Sup. 1991

\bibitem{LQPP}
H. Amann. \textit{Linear and quasilinear parabolic problems,
{Volume} {I}: Abstract linear theory.} Birkh\"auser, Basel,
Boston, Berlin 1995.

\bibitem{AmannEscherII}
H. Amann, J. Escher. {\it Analysis II.} Birkh\"auser, Basel 1999.


\bibitem{ClementEtAl} 
Ph. Cl{\'e}ment, H.J.A.M. Heijmans, S. Angenent, C.J. van Duijn, B. de Pagter. {\it One-parameter semigroups.} CWI Monographs, 5. North-Holland Publishing Co., Amsterdam, 1987.

\bibitem{CrandallRabinowitz}
M.C. Crandall, P.H. Rabinowitz. {\it Bifurcation from simple eigenvalues.} J. Functional Analysis {\bf 8} (1971), 321-340. 

\bibitem{Cushing1}
J.M. Cushing. {\it Existence and stability of equilibria in age-structured
population dynamics.} J. Math. Biology {\bf 20} (1984), 259-276.

\bibitem{Cushing2}
J.M. Cushing. {\it Global branches of equilibrium solutions of the McKendrick equations for age structured population growth.} Comp. Math. Appl. {\bf 11} (1985), 175-188.

\bibitem{Cushing3}
J. Cushing. {\it Equilibria in structured populations.} J. Math. Biology {\bf 23} (1985), 15-39.


%\bibitem{CushingBook}
%J.M. Cushing. {\it An introduction to structured population dynamics.} CBMS-NSF Regional Conference %Series in Applied Mathematics. SIAM 1998.

\bibitem{DanersKochMedina}
D. Daners, P. Koch-Medina. {\it Abstract Evolution Equations, Periodic Problems, and Applications.}
Pitman Res. Notes Math. Ser., {\bf 279}, Longman, Harlow 1992. 

\bibitem{DelgadoEtAl}
M. Delgado, M. Molina-Becerra, A. Su\'arez. {\it A nonlinear age-dependent model with spatial diffusion.}
J. Math. Anal. Appl. {\bf 313} (2006), 366-380. 

\bibitem{DelgadoEtAl2}
M. Delgado, M. Molina-Becerra, A. Su\'arez. {\it Nonlinear age-dependent diffusive equations: A bifurcation approach.}
J. Diff. Equations {\bf 244} (2008), 2133-2155. 

\bibitem{Pruess2}
J. Pr\"u\ss . {\it On the qualitative behaviour of populations with age-specific interactions.} Comput.
Math. Appl. {\bf 9} (1983), 327-339.

\bibitem{RBA}
P.H. Rabinowitz. {\it Some global results for nonlinear eigenvalue problems.} J. Functional Analysis {\bf 7} (1971), 487-513.

\bibitem {Triebel}
H. Triebel. \textit{Interpolation theory, function spaces,
differential operators.} Second edition. Johann Ambrosius Barth. Heidelberg, Leipzig 1995.

\bibitem{WalkerEJAM}
Ch. Walker. {\em Global existence for an age and spatially structured haptotaxis model with nonlinear age-boundary conditions.}
Europ. J. Appl. Math. {\bf 19} (2008), 113-147.

\bibitem{WalkerDCDS}
Ch. Walker. {\em Age-dependent equations with nonlinear diffusion.}
To appear in: Discrete Contin. Dyn. Syst. A.

\bibitem{WalkerSIMA}
Ch. Walker. {\em Positive equilibrium solutions for age and spatially structured population models.}
SIAM J. Math. Anal. {\bf 41} (2009), 1366-1387.

\bibitem{WalkerJDE}
Ch. Walker. {\em Global bifurcation of positive equilibria in nonlinear population models.}
To appear in: J. Diff. Eq.

\bibitem{WebbSpringer}
G.F. Webb. {\em Population models structured by age, size, and spatial position.} In: P. Magal, S. Ruan (eds.) {\em Structured population models in biology and epidemiology.} Lecture Notes in Mathematics, Vol. 1936. Springer, Berlin, 2008.

\end{thebibliography}
\end{document}